\documentclass{article}
\usepackage{amsmath,amsfonts,amsthm,amssymb,amscd,color,xcolor,mathrsfs,verbatim}
\usepackage{graphicx,eurosym}
\usepackage{hyperref}
\usepackage[applemac]{inputenc}
\usepackage[cyr]{aeguill}
\colorlet{darkblue}{blue!50!black}
\hypersetup{
    colorlinks,%
    citecolor=blue,%
    filecolor=red,%
    linkcolor=darkblue,%
    urlcolor=blue,%
    pdfnewwindow=true,%
    pdfstartview={FitH}
}

\DeclareMathAlphabet\mathbfcal{OMS}{cmsy}{b}{n}

\newcommand{\aA}{{\cal A}}
\newcommand{\KK}{{\cal K}}

\newcommand{\LL}{{\cal L}}

\newcommand{\VV}{{\cal V}}
\newcommand{\ty}{\infty}

\newcommand{\lag}{\langle}
\newcommand{\rag}{\rangle}
\newcommand{\la}{\lambda}
\newcommand{\DD}{{\cal D}}

\newcommand{\I}{{\mathbb I}}

\newcommand{\IP}{{\mathbb P}}

\newcommand{\BBBBB}{{\mathcal B}}

\newcommand{\e}{\varepsilon}
\newcommand{\DDD}{{\boldsymbol{D}}}

\newcommand{\EE}{{\cal E}}

\newcommand{\RR}{{\cal R}}

\newcommand{\be}{\begin{equation}}
\newcommand{\ee}{\end{equation}}
\newcommand{\bp}{\begin{proof}}
\newcommand{\ep}{\end{proof}}
\newcommand{\bi}{\begin{itemize}}
\newcommand{\ei}{\end{itemize}}

\newcommand{\PPPP}{\mathfrak P}

 \newcommand{\HH}{\mathcal{H}}

\newcommand{\R}{\mathbb{R}}

\newcommand{\PP}{{\cal P}}

\newcommand{\p}{\partial}

\newcommand{\diver}{\mathop{\rm div}\nolimits}

\newcommand{\supp}{\mathop{\rm supp}\nolimits}

\newcommand{\es}{\varepsilon}

\newcommand{\dd}{\,{\textup d}}

\theoremstyle{plain}
\newtheorem{theorem}{Theorem}[section]
\newtheorem*{mt}{Main Theorem}
\newtheorem{lemma}[theorem]{Lemma}
\newtheorem{proposition}[theorem]{Proposition}

\theoremstyle{definition}

\theoremstyle{remark}

\numberwithin{equation}{section}
\newtheorem*{definition*}{Definition}
\newtheorem*{problem*}{Problem}

\newtheorem*{remark*}{Remark}
\newtheorem*{note*}{Note}
\begin{document}
\author{Vahagn~Nersesyan\footnote{Laboratoire de Math\'ematiques, UMR CNRS 8100, UVSQ, Universit\'e Paris-Saclay, 45,~av.~des Etats-Unis, F-78035 Versailles, France \& Centre de Recherches Math\'emati\-ques,~CNRS~UMI 3457, Universit\'e de Montr\'eal, Montr\'eal,  QC, H3C 3J7, Canada;    e-mail: \href{mailto:Vahagn.Nersesyan@math.uvsq.fr}{Vahagn.Nersesyan@math.uvsq.fr}}}

\title{Ergodicity   for the randomly forced   Navier--Stokes system in a two-dimensional unbounded domain}

\date{\today}

 \maketitle

\begin{abstract}

 The ergodic properties of the randomly forced   Navier--Stokes system have been extensively studied in the literature during the last two decades.    The problem has always been considered in bounded domains,     in order to have, for example,  suitable spectral properties for the Stokes operator,   to ensure some compactness properties for the resolving operator of the system and the associated functional spaces,
    etc.~In the present paper, we  consider  the   Navier--Stokes system  in an unbounded domain satisfying the Poincar\'e inequality.~Assuming that the system is perturbed by a bounded non-degenerate noise, we
  establish uniqueness of  stationary measure and exponential mixing in the dual-Lipschitz metric.~The proof is carried out by developing the controllability approach of the papers~\cite{shirikyan-2018,KNS-2018} and using the asymptotic compactness of the   dynamics.

 \smallskip  
\noindent
{\bf AMS subject classifications:}    35Q30, 37A25, 93B05.

\smallskip
\noindent
{\bf Keywords:}    Navier--Stokes system, unbounded domain, exponential mixing,   controllability, asymptotic compactness. 
\end{abstract}

     \tableofcontents

\setcounter{section}{-1}

 \section{Introduction}

  The ergodicity of    randomly forced 2D Navier--Stokes (NS) system   has been widely studied in the literature in the case of   bounded domains (see the papers  \cite{FM-1995,KS-cmp2000,EMS-2001,BKL-2002} for the   first  results and the book    \cite{KS-book} and the 
  reviews   \cite{MR2459085, debussche-2013}  
   for a detailed discussion of different          methods and for further references).~This paper is concerned with the ergodic behaviour of  the NS  system   in an {\it unbounded}~domain $D$ in $\R^2$ with smooth boundary $\p D$:    
  \begin{gather}
\p_t  u-\nu \Delta u+\lag u, \nabla \rag u+\nabla p   = \eta(t,x),  \quad x\in D,
\quad    \label{0.1}\\
 \diver u   =0, \quad  u|_{\p D}=0, \label{0.2}\\
u(0) =u_0.\label{0.3}
\end{gather}
Here   $\nu>0$ is the kinematic viscosity of the fluid, $u = (u_1(t,x), u_2(t,x))$ is the velocity field,~$p = p(t, x)$ is the pressure, and~$\eta$ is  an    external random force. To have a suitable dissipativity property for   solutions, we 
  assume
  that   $D$ is a {\it Poincar\'e domain},\footnote{E.g.,  we can assume that $D$ is  bounded in some direction $a\in \R^2_*$, i.e., $\sup_{x\in D} |\lag a,x \rag|<\ty $.} i.e., there is a number $\la_1>0$ such that
\begin{equation}\label{0.4}
	\int_D |v|^2  \dd x\le \la_1^{-1}\int_D |\nabla v|^2  \dd x, \quad v \in C_0^\ty(D,\R^2).
\end{equation} The random force $\eta$ is    a   process of the~form  
\begin{equation} \label{0.7}
\eta(t,x)=\sum_{k=1}^\infty \I_{[k-1,k)}(t)\eta_k(t-k+1,x), \quad t\ge0, \,\, x\in D,
\end{equation}
where  $\I_{[k-1,k)}$ is the indicator function of the   interval~$[k-1,k)$ and $\{\eta_k\}$ is a sequence of   independent and identically distributed (i.i.d.) random variables~in~the space\footnote{We denote by $H$     the usual  functional space  for the NS system     defined by~\eqref{N:0.9}.} 
     $E:=L^2([0,1],H)$.~Moreover, the law of $\eta_k$ is assumed to be {\it decomposable} in the following~sense. 
     \begin{description}
\item[Decomposability.]{\sl There is an   orthonormal basis~$\{e_j\}$     in~$E$ such~that
\begin{equation} \label{0.8}
\eta_k=\sum_{j=1}^\infty b_j \xi_{jk}e_j 
\end{equation}
for some real-valued  independent~random variables $\xi_{jk}$ verifying $|\xi_{jk}|\le1$ and some
 positive numbers~$b_j$ such that $\sum_{j=1}^\ty b_j^2<\ty.$~The   law of the random variable~$\xi_{jk}$~is absolutely continuous  with respect to the Lebesgue measure and the corresponding density~$\rho_j$ is       $C^1$-smooth  and~$\rho_j(0)>0$ for all~$j\ge1$.

 } 
\end{description}
The restriction to integer times of the  velocity field~$u_t$       defines a family of Markov processes~$(u_k, \IP_u)$ parametrised by the initial condition  $u_0= u \in H$.~The associated Markov operators are denoted by~$\PPPP_k$ and~$\PPPP_k^*$.~Recall that a measure~$\mu\in \PP(H)$ is stationary for the family~$(u_k, \IP_u)$ if $\PPPP_1^*\mu=\mu$.
  In this paper, we prove the following~result.  
\begin{mt}Under the above assumptions, the family~$(u_k, \IP_u)$ has a unique stationary measure $\mu\in\PP(H)$.~Moreover, it is exponentially mixing in the following sense: for any compact set $\HH\subset H$, there are numbers
  $C>0$ and~$c>0$ such~that 
\begin{equation}\label{mix}
\|\PPPP_k^*\la-\mu\|_{L(H)}^*
\le C  e^{-c k}, \quad    k\ge  0
\end{equation}for any initial measure  $\la\in \PP(H)$ with $\supp\la\subset \HH$.~Here
   $\|\cdot\|_{L(H)}^*$ is the dual-Lipschitz metric defined by~\eqref{N:0.12}.  
 \end{mt}  
  To the best of our knowledge,
    this  is the first result that establishes  uniqueness of stationary measure and exponential mixing for the NS system in an unbounded domain.~In this unbounded setting there are at least two additional difficulties  compared to the case of a bounded domain.~First, the resolving   operator  of the equation is  not compact and does not have a compact attracting set.~The second and more important difficulty   is related to the presence of a continuous component in the spectrum  of the
Stokes operator.~Let us note that the existence of a stationary measure for the NS system perturbed by the random force~\eqref{0.7} can be established by combining the  {\it Bogolyubov--Krylov argument} and the   {\it asymptotic compactness} of the dynamics.~When the driving force is a white noise, 
        existence results are obtained  in~\cite{MR1808628},   for a real Ginzburg--Landau  equation, in~\cite{MR1889231},
   for  a complex Ginzburg--Landau equation,  in~\cite{MR2238928}, for the NS system, and in~\cite{MR3689215}, for a damped Schr\"odinger equation.


 The     uniqueness of stationary measure and  mixing properties  for PDEs in unbounded domains have been studied previously only for the Burgers equation.~The case of  inviscid  equation on the   real line  has been   considered  in the paper~\cite{MR3110798},   under the assumption that the  driving force is  a space-time homogeneous Poisson point process.~The proof is based on a combination of   Lagrangian methods and first/last passage percolation theory.~This result is generalised to the   viscous case in the recent  paper~\cite{BL-2019}, where the
    perturbation   is a space-time homogeneous random kick force.~In both papers, the stationary measure is space translation invariant. In the case of the   NS   system perturbed by a {\it homogeneous} noise, 
    the uniqueness of a space-time translation invariant measure (and in some cases also the existence)   remains   an   {\it open~problem}.

  The proof of the Main Theorem uses
   a controllability approach         of the papers~\cite{shirikyan-asens2015, shirikyan-2018}, where     exponential mixing is established  for the NS~system with   a space-time and  boundary {\it localised}   forcing.~In~\cite{KNS-2018, KNS-2019},  methods of control theory are used   to study    a family of parabolic PDEs   with a random perturbation  that is {\it highly degenerate}    in the Fourier space; see also the paper~\cite{KZ-2019}, where a    non-degenerate version of these results is presented. In~\cite{JNPS-2019}, controllability is used to derive large deviations principle for the Lagrangian trajectories of the~NS~system.

    The main novelty of the controllability  argument we use  here  is that the deterministic operator is not supposed to be regularising  to a space that is compactly embedded into the main phase space (see Theorem~\ref{T:1.1}).~That regularisation condition is replaced by two weaker   properties: we assume that  the nonlinear dynamics   is asymptotically compact and the   linearised operator         can be decomposed into a sum of two   operators  one of which is {\it dissipative} and the other   is {\it compact} (see~\eqref{reps1}).~It is proved   that  these  weaker~conditions are still sufficient for the exponential mixing.~We make an essential use~of the   form of the nonlinearity      to verify these    properties for   the NS system.~The reader is referred to Section~\ref{S:1.1} for   further  discussion of     the abstract controllability~criterion and to Section~\ref{S:2} for the verification of the conditions of this criterion for the~NS~system.

    Let us close this section with two remarks.~Note that the convergence rate~$c$ in~\eqref{mix} depends on~the~initial compact~$\HH$.~Indeed, this is related to the fact that  the NS system in the unbounded case does not have   a compact invariant attracting set.~Without going into the details, let us mention   that the existence of such attracting set (and a uniform  convergence rate) can be recovered  if we take
  initial condition  and   forcing   in   weighted Sobolev spaces as in~\cite{B-1992}.

The second remark is about the   NS system with the  {\it Ekman damping}.~By literally repeating the arguments of the   proof of the Main Theorem, one can    establish exponential mixing   in the case of the whole space $D=\R^2$ or arbitrary  unbounded\footnote{I.e.,   domain $D$ which does not necessarily satisfy the Poincar\'e inequality.} domain~$D\subset \R^2$  with smooth boundary   for the following     system:  
   \begin{gather*}
\p_t  u-\nu \Delta u+au+\lag u, \nabla \rag u+\nabla p   = \eta(t,x), 
\quad    \\
 \diver u   =0, \quad u|_{\p D}=0,
 \end{gather*}where $a>0$ is the damping parameter.~The damping   ensures a   dissipativity property for the solutions without any assumption on the   domain. 
  We shall not discuss the details of this generalisation in this paper.

          \subsection*{Acknowledgments}

         The author   thanks   Armen Shirikyan for   helpful discussions and suggestions.~This research  was supported by Agence Nationale de la Recherche   grant    ANR-17-CE40-0006-02 and CNRS PICS grant Fluctuation theorems in stochastic systems.

\subsection*{Notation}
 
Let   $D\subset \R^2$ be an unbounded Poincar\'e domain   with smooth boundary. In this paper,
 we use the following functional spaces. 
 
  \smallskip
\noindent 

 \smallskip
\noindent 
$C_0^\ty(D,\R^2)$ is the space of compactly supported smooth functions $u:D\to \R^2$, and  
\begin{equation}\label{N:0.8}
\VV=\{u\in C_0^\ty(D,\R^2):\diver u=0\}.
\end{equation}

 \smallskip
\noindent 
$H^s(D, \R^2)$ and  $L^p(D, \R^2)$   are the   Sobolev and Lebesgue  spaces on $D$.~We consider the NS system in the usual  
    functional spaces:
\begin{align}
H&=\text{closure of $\VV$ in $L^2(D,\R^2)$,}\label{N:0.9}	\\
V&=\text{closure of $\VV$ in $H^1(D,\R^2)$}	\label{N:0.10}
\end{align}   endowed   with the scalar products
\begin{gather*}
\lag u,v\rag=\int_D u\cdot v\dd x, \quad\quad \lag u,v\rag_1=\int_D (\nabla u_1\cdot \nabla v_1+\nabla u_2\cdot \nabla v_2)\dd x
\end{gather*}
and the corresponding norms $\|\cdot\|=\sqrt{\lag\cdot,\cdot\rag}$ and~$\|\cdot\|_1=\sqrt{\lag\cdot,\cdot\rag_1}$. The Poincar\'e inequality \eqref{0.4} implies that $\|\cdot\|_1$ is equivalent to the~norm~inherited from~$H^1(D,\R^2)$.  
The       dual of~$V$  with respect to the scalar product in $H$ is denoted by $V'$.

 \smallskip
   Let $X$ be a   Polish space with metric $d$.

\smallskip
\noindent
 ${\cal B}(X)$ denotes the Borel $\sigma$-algebra on $X$. 
 
   \smallskip
\noindent
$C_b(X)$    is the space of continuous functions $g:X\to \R$ endowed with the sup-norm    $\|g\|_{\ty}$.~When $X$ is compact, we write $C(X)$.

  \smallskip
\noindent
$L_b(X)$   is the space of     functions $g\in C_b(X)$   for which the following norm is finite: 
$$
 \|g\|_{L(X)}=\|g\|_\infty+\sup_{u\neq v} \frac{
{|g(u)-g(v)|}}{ d(u,v)  }.
$$ 

\smallskip
\noindent
   $\PP(X)$ is the set of Borel probability measures on $X$ endowed with the   metric  
\begin{equation}\label{N:0.12}
  \|\mu_1-\mu_2\|_{L(X)}^*=\sup_{\|g\|_{L(X)}\le 1} |\lag g,\mu_1\rag-\lag g,\mu_2\rag|, \quad \mu_1, \mu_2\in \PP(X),
\end{equation}
  where  $\lag g,\mu\rag=\int_Xg(u)\mu(\!\dd u)$.

    \smallskip
 Let  $E$ be a Banach  space endowed with a
norm $\|\cdot\|_E$, and let $J_T=[0, T ]$.    

  \smallskip
\noindent 
$L^p(J_T,E)$,   $1\leq p<\infty$  is the
space of measurable functions $u: J_T \rightarrow E$ such~that
\begin{equation}
\|u\|_{L^p(J_T,E)}=\bigg(\int_{0}^T \|u(s)\|_E^p\dd s
\bigg)^{\frac{1}{p}}<\infty.\nonumber
\end{equation}
\noindent 
$C(J_T,E)$ is the space of continuous functions $u:J_T\to E$ with the 
norm 
$$
\|u\|_{C(J_T,E)}=\sup_{t\in J_T} \|u(t)\|_E.
$$ 
$\LL(E,Y)$ is the space of bounded linear operators from $E$ to another Banach
space $Y$.~We write $\LL(E)$ when $E=Y$.

  \smallskip
\noindent
    $B_E(a,R)$   is the closed   ball in~$E$ of radius~$R$ centered at~$a$.~We  write  $B_E(R)$ when     $a=0$.  
    
 \smallskip
\noindent
$\DD(\eta)$ is the law of $E$-valued random variable $\eta$.

      \section{Abstract criterion}\label{S:1.1}

Let $H$ be a separable  Hilbert space, and $E$ be a separable Banach space.   In this section, 
 we consider a   {\it random dynamical system}  of the form
\begin{equation}\label{E:1.1}
u_k=S(u_{k-1}, \eta_k), \quad k\ge1,	
\end{equation}where      $S:H\times E\to H $  is a continuous   mapping and $\{\eta_k\}$ is a sequence of i.i.d.~random variables in $E$.~Let $\KK\subset E$  be the support of   
  the law~$\ell:=\DD(\eta_k)$. For any     sequence        $ \{\zeta_k\}$ in  $\KK$,  let us denote  by  $ S_k(u; \zeta_1, \ldots, \zeta_k)$   the trajectory of~\eqref{E:1.1}  corresponding to   the initial condition
  \begin{equation}\label{E:1.2} 
u_0=u\in H
\end{equation} and the vectors $\eta_i=\zeta_i$, $i=1, \ldots, k$.~For any     set~$\HH\subset H$,  we  define  the    {\it set of attainability  in time $k\ge1$}:      
$$
 \aA_k(\HH):=\{S_k(u; \zeta_1,\ldots,\zeta_k): u\in \HH,\,\, \zeta_1,\ldots,\zeta_k\in \KK\} 
 $$and the {\it set of attainability  in infinite time}: 
 $$ 
\aA(\HH):=\overline{\cup_{k=1}^\ty \aA_k(\HH)}^H.
$$  The following five  conditions are assumed to be satisfied   for  the mapping   $S$ and the  measure~$\ell$.

 \medskip
{\bf (i)~Regularity.}{\sl~The mapping $S:H\times E\to H$ is twice continuously differentiable, and its   derivatives are bounded on bounded subsets.~Moreover, for any~$(u,\eta)\in H\times  \KK$, the  derivative~$(D_uS)(u,\eta)$ can be represented as  
\begin{equation}\label{reps1}
	(D_uS)(u,\eta)=\Psi_1+\Psi_2(u,\eta),  
\end{equation}where the operators\footnote{$\Psi_1$ does not depend on $(u,\eta)$.} $\Psi_1, \Psi_2(u,\eta)\in L(H)$ are such that
\begin{equation}
  \|\Psi_1\|_{\LL(H)}  = :\varkappa <1  \label{reps2}
 \end{equation} and      $\Psi_2(u,\eta)$ is  compact.

 \medskip

{\bf (ii)~Asymptotic compactness.} {\it For any bounded sequence $\{u_0^n\}$ in $H$, any   integers $l_n\ge1$ such that $l_n\to \ty$, and any family $\{\zeta_m^n: m,n\ge1\}\subset \KK$, the sequence $\{S_{l_n}(u_0^n; \zeta_1^n,\ldots,\zeta_{l_n}^n)\}$ is precompact in $H$.}

 \medskip

{\bf (iii)~Approximate controllability to a point.}{\sl~There is a point $\hat u\in H$ with the following property:~for any $\e>0$ and any compact $\HH$ in $H$,   there is an~integer $n\ge1$ such that, for any initial point $u\in \HH$, there are  vectors~$\zeta_1, \ldots, \zeta_n\in \KK$ satisfying  }
\begin{equation}\label{1.3}
\| S_n(u; \zeta_1, \ldots, \zeta_n)- \hat u\|\le \es. 
\end{equation}

 {\bf (iv)~Approximate controllability of the linearisation.}{\sl~For any  $u\in H$ and $\eta\in\KK$,   the image of the linear mapping $(D_\eta S)(u,\eta):E\to H$ is dense~in~$H$.}

   \medskip

 {\bf (v)~Decomposability.}{\sl~The set $\KK$ is compact in $E$.~Moreover, there are   sequences of closed subspaces~$\{F_n\}$ and $\{G_n\}$ in~$E$ satisfying the following properties:}
\begin{itemize}{\sl
\item[$\bullet$]$F_n$ are finite-dimensional, 
  $F_n\subset F_{n+1}$ for any $n\ge1$, and  $E=\overline{\,\cup_nF_n\,}$.
\item[$\bullet$]
$E$ is the direct sum of the spaces~$F_n$ and~$G_n$, and the norms of the corresponding projections~${\mathsf P}_n$ and~${\mathsf Q}_n$   are bounded    in~$n\ge 1$. 
\item[$\bullet$]
  $\ell$ is the product of   projections ${\mathsf P}_{n*}\ell$ and~${\mathsf Q}_{n*}\ell$ for   any $n\ge 1$.~Moreover, the measures~${\mathsf P}_{n*}\ell$ have   $C^1$-smooth  densities  with respect to the Lebesgue measure~on~$F_n$.}
\end{itemize}
}

 System~\eqref{E:1.1} defines a family of Markov processes~$(u_k, \IP_u)$ in~$H$  parametrised by  the initial condition \eqref{E:1.2}. Let $\PPPP_k:C_b(H)\to C_b(H)$ and $\PPPP_k^*:\PP(H)\to \PP(H)$   be the associated  Markov operators. 
 \begin{theorem}\label{T:1.1}  
Under   Conditions {\rm(i)-(v)}, the family $(u_k,\IP_u)$
  has a unique stationary measure~$\mu\in\PP(H)$.~Moreover,
  for any compact set $\HH$ in $H$, there are numbers
  $C>0$ and~$c>0$ such~that 
\begin{equation}\label{E:mix}
\|\PPPP_k^*\la-\mu\|_{L(H)}^*
\le C  e^{-c k}, \quad    k\ge  1
\end{equation}for any initial measure  $\la\in \PP(H)$ with $\supp\la\subset \HH$.
  \end{theorem}
  See the papers~\cite{shirikyan-asens2015,shirikyan-2018, KNS-2018, KNS-2019, KZ-2019, JNPS-2019} for related abstract  criteria for uniqueness of stationary measure and mixing.~The formulation of Theorem~\ref{T:1.1} is close to the results in~\cite{KNS-2019, JNPS-2019}, but the proof is based on a theorem obtained in~\cite{shirikyan-2018}.~There are two main  differences  in our formulation.~First,  in Condition~(i), we do not suppose   that the mapping~$S$ (or its linearisation) takes values in a space   that is compactly embedded~into~$H$. Instead, we    assume that~$(D_uS)(u,\eta)$ is~a~sum of    {\it dissipative}~and    {\it compact} operators.~The second difference is   
Condition~(ii), which allows to recover some compactness properties\footnote{The asymptotic  compactness is a well-known property in the study of the attractors for  deterministic PDEs  (e.g., see~\cite{MR1133627,  MR1491614}).} for the dynamics.~These two new conditions make  Theorem~\ref{T:1.1} applicable to the NS system  in unbounded~domains.~Moreover, this theorem can be applied to dissipative PDEs without parabolic regularisation (this  will be addressed in a
subsequent publication). 
   
    As in this paper the random perturbation is {\it non-degenerate}, the image of~the mapping $(D_\eta S)(u,\eta)$ is assumed to be dense   for any     $\eta$ in $\KK$;   degenerate versions of this criterion can also be envisaged    (cf. \cite{KNS-2018, KNS-2019}).   
  \begin{proof}[Proof of Theorem \ref{T:1.1}] 
{\it Step 1: Existence of stationary measure.}~The set~$X:=\aA(\HH)$ is compact in $H$.~Indeed, 
it suffices to show that   any sequence $\{v_n\}$   of the~form
$$
v_n=S_{l_n}(u_0^n;\zeta_1^n,\ldots,\zeta_{l_n}^n)
$$
with some  $u_0^n\in \HH$,  $\zeta_1^n,\ldots,\zeta_{l_n}^n\in \KK$, and   $l_n\ge1$, is precompact  in~$H$.~If~the sequence $\{l_n\}$ is bounded, then~$v_n\in \aA_m(\HH)$ for all $n\ge1$, where~$m=\max \{l_n\}$.~The fact  that $\aA_m(\HH)$ is compact (as image of a compact set by a continuous mapping) implies that $\{v_n\}$ is precompact.~In the case   $l_n\to \ty$, the conclusion follows from Condition~(ii).

 The compactness of $X$, combined with the invariance property~$S(X\times \KK)\subset X$ and the usual Bogolyubov--Krylov argument (e.g., see Section~2.5 in~\cite{KS-book}),  implies the existence of a stationary measure $\mu\in \PP(X)$.

 \smallskip
{\it Step 2: Limit~\eqref{E:mix}.}~According to        Theorem~1.1 in~\cite{shirikyan-2018}, limit~\eqref{E:mix}  will be established if we   verify the   following property. 
        
        \medskip
        
         {\bf  Local stabilisation.}{\sl~Let $\DDD_\delta:=\{(u,u')\in X\times X: \|u-u'\|\le\delta\}$.~For any $R>0$ and any compact~$\KK\subset E$, there is a finite-dimensional subspace~$\EE\subset  E$, and a continuous mapping 
$$
\varPhi:\DDD_\delta\times B_E(R)\to \EE, \quad (u,u',\eta)\mapsto \eta', 
$$ 
which is continuously differentiable in~$\eta$ and satisfies the inequalities
\begin{gather}
\sup_{\eta\in B_E(R)}
\left(\|\varPhi(u,u',\eta)\|_E+\|D_\eta\varPhi(u,u',\eta)\|_{\LL(E)}\right)
\le C\,\|u-u'\|,\label{E:1.5}\\
\sup_{\eta\in \KK}\| S(u,\eta)-S(u',\eta+\varPhi(u,u',\eta))\|
\le q\,\|u-u'\|, \quad (u,u')\in \DDD_\delta \label{E:sdfdfge}
\end{gather}
  for some positive constants~$C$, $\delta$, and~$q<1$.}
  
Let us show that   Conditions~{\rm(i)},~{\rm(iv)}, and (v) imply this local stabilisation property.~We use a construction  of approximate   right inverse for linear operators   from Section~2.2 in~\cite{KNS-2018}.~For any $u\in X$ and~$\eta\in E$, let   $A(u,\eta):E\to H $ be given by~$A(u,\eta):=(D_\eta S)(u,\eta)$.~Then
 $G(u,\eta):=A(u,\eta)\,A(u,\eta)^*:H\to H$ is non-negative self-adjoint operator  	and~$\text{Im} (G(u,\eta))$ is dense in $H$ by Condition~{\rm(iv)}. Thus $(G(u,\eta)+\gamma I)^{-1}$ is well defined for any~$\gamma>0$, and   we have the~limit
\begin{equation}\label{haks}
G(u,\eta)(G(u,\eta)+\gamma I)^{-1}f\to f\quad\mbox{as $\gamma\to0^+$}, \quad f\in H,
\end{equation}
by Lemma~2.4 in~\cite{KNS-2018}.~This shows that $A(u,\eta)^*(G(u,\eta)+\gamma I)^{-1}$ is an approximate 
right inverse for $A(u,\eta)$. We truncate it to obtain an operator with  finite-dimensional image:
$$
 \RR_{M,\gamma}(u,\eta)
:={\mathsf P}_MA(u,\eta)^*(G(u,\eta)+\gamma I)^{-1}, 
$$ where ${\mathsf P}_M$ is a projection as in Condition~(v)  and 
 $M\ge1$ and~$\gamma>0$  are~parameters that will be chosen later. It is straightforward to see that
\begin{equation} 
\bigl\| \RR_{M,\gamma}(u,\eta)\bigr\|_{\LL(H,E)}+\bigl\|(D_\eta  \RR_{M,\gamma})(u,\eta)\bigr\|_{\LL(H\times E,E)} \le C_1(R,M,\gamma)    \label{E:1.7}
\end{equation}  for any~$u\in X$ and $\eta\in B_E(R)$, where the constant $C_1(R,M,\gamma)>0$ does not depend on~$(u,\eta)$.~By the Taylor  formula, for any $u,u'\in X$ and $\eta,\eta'\in B_E(R)$,  we have
\begin{equation} \label{E:1.8}
S(u',\eta')-S(u,\eta)=(D_u S)(u,\eta)(u'-u)+(D_\eta S)(u,\eta)(\eta'-\eta)
+r(u,u',\eta,\eta') 
\end{equation}
where   
\begin{equation} \label{E:1.9}
\|r(u,u',\eta,\eta')\|\le C_2(R)\,\bigl(\|u-u'\|^2+\|\eta-\eta'\|_E^2\bigr). 
\end{equation}
  The mapping $\varPhi$ is defined by 
$$
\varPhi(u,u',\eta):=-  \RR_{M,\gamma}(u,\eta)\Psi_2(u,\eta)(u'-u). 
$$Then \eqref{E:1.5} is verified due to    \eqref{E:1.7} and   
\begin{equation}\label{psi2}
	  C_3:=\sup_{(u,\eta)\in X\times  B_E(R)}\left(\|\Psi_2(u,\eta)\|_{\LL(H)}+\|D_\eta\Psi_2(u,\eta)\|_{\LL(H\times E,H)} \right)   <\ty.   \end{equation}
 The fact that $C_3$ is finite follows from the   boundedness on $X\times  B_E(R)$ of the norms of the derivatives of  $  S(u,\eta)$  and the representation~\eqref{reps1}.~Assume that,   for any~$\e>0$, we are able to find    numbers $M\ge1$   and~$\gamma>0$ such that 
 \begin{equation}
\bigl\|(D_\eta S)(u,\eta)	\RR_{M,\gamma} (u,\eta)\Psi_2(u,\eta)g-\Psi_2(u,\eta)g\bigr\|   \le \e\, \|g\| \label{E:1.6} 
\end{equation}   for any $u\in X$, $\eta\in  \KK$, and  $g\in H$.~Then combining  \eqref{reps1}, \eqref{reps2}, and \eqref{E:1.8}-\eqref{E:1.6}, we~obtain
\begin{align*}  
\|S(u,\eta)&-S(u',\eta+\varPhi(u,u',\eta))\| \\
&\qquad \le  \varkappa\, \|u'-u\| +\e\,\|u'-u\|+\|r(u,u',\eta,\eta+\varPhi(u,u',\eta))\|\nonumber\\
&\qquad \le (\varkappa+\e + C_4(\e,R)\,\|u'-u\|  )\|u'-u\| 
\le q\,\|u'-u\|   
\end{align*} for any  $(u,u')\in \DDD_\delta$ and $\eta\in \KK$,
 where  $\e>0$ and $\delta>0$ are sufficiently small  and~$q<1$. 
 This completes the proof of  the  local stabilisation and limit~\eqref{E:mix}.  
  
 Let us  prove inequality \eqref{E:1.6}.~By~\eqref{haks},   the continuous dependence of the operators~$(D_\eta S)(u,\eta)$ and~$\Psi_2(u,\eta)$ on $(u,\eta)$, and  a simple   compactness argument, we find a large integer~$M\ge1$   and  a small number  $\gamma>0$ such that\footnote{Here we use  also the boundedness of the projections~${\mathsf P}_M$ (see Condition~(v)), which implies the limit ${\mathsf P}_M\to I$ as $M\to \ty$ in the   operator topology.} 
$$
\left\|(D_\eta S)(u,\eta)	\RR_{M,\gamma} (u,\eta)f-f\right\|   \le \e  
$$  for any $u\in X$, $\eta\in  \KK$, and  $f\in   \Psi_2(u,\eta)(B_H(1))$.~Then inequality \eqref{E:1.6} follows by   homogeneity.

\smallskip
{\it Step 3: Uniqueness of stationary measure.}~To complete the proof, it remains to show the uniqueness of   stationary measure in~$\PP(H)$.~Assume that $\mu   $ is the  stationary measure supported\footnote{Using Condition~(iii), it is easy to see that   $\supp \mu =\aA(\{\hat u\})$.} in $\aA(\{\hat u\})$, and let~$\la_1$  be   any stationary measure for $(u_k,\IP_u)$ in~$\PP(H)$.~As $\la_1(\aA(H))=1$,  there is a sequence $\{\HH_n\}$ of compacts  in $H$  such that 
$$
\la_1(\aA(\HH_n))>1-1/n \quad  \text{for any $n\ge1$}.
$$ By Condition~(iii), we have  $\aA(\{\hat u\})\subset \aA(\HH_n)$, and
$\mu$ and~$\la_1/ \la_1(\aA(\HH_n))$ are stationary measures for~$(u_k,\IP_u)$ in $\PP(\aA(\HH_n))$.~So~$\mu=\la_1/ \la_1(\aA(\HH_n))$ by the uniqueness of stationary measure on~$X=\aA(\HH_n)$. Therefore, 
$$
\la_1 (\Gamma)=\lim_{n\to \ty} \la_1 (\Gamma \cap \aA(\HH_n))=\lim_{n\to \ty} \la_1 (\Gamma \cap \aA(\HH_n))/ \la_1(\aA(\HH_n)) =\mu (\Gamma)
$$  for any $ \Gamma\in \BBBBB(H)$. Thus $\la_1=\mu$.
 \end{proof}

 \section{Proof of the Main Theorem}\label{S:2}
 
  In this section, we prove the Main Theorem by applying  Theorem~\ref{T:1.1}.~We   begin with a short discussion  of the deterministic   NS system,
  then turn to the verification of Conditions {\rm(i)-(v)} in an appropriate functional  setting.

 \subsection{Preliminaries}
 
Applying the Leray  projection\footnote{That is the orthogonal projection in $L^2(D,\R^2)$ onto $H$.} $\Pi$ to Eq.~\eqref{0.1}, we eliminate the pressure term  and consider the   evolution~system  
 \begin{equation} 
\dot u +  \nu Lu + B(u) = \eta,\label{1.1}
\end{equation} where  $L=-\Pi\Delta$ is the Stokes operator and $B(u)=\Pi (\lag u, \nabla \rag u)$.

 Let us define a bilinear   symmetric form $[\cdot,\cdot]:V\times V\to \R$ by
$$
[u,v]:=   \lag u,v\rag_1-  \frac{\la_1}{2}\lag u,v\rag, \quad u,v\in V.
$$The      Poincar\'e inequality (see~\eqref{0.4})   
\begin{equation}\label{0.9}
\|u\|^2\le \la_1^{-1} \|u\|_1^2,\quad u\in V
\end{equation}
 implies that  
\begin{equation}\label{0.10}
	\frac{1}{2}\|u\|_1^2 \le [u]^2:=[u,u]\le  \|u\|^2_1.
\end{equation}Thus $[\cdot,\cdot]$ defines a scalar product on $V$ with norm $[\cdot]$ equivalent to $\|\cdot\|_1$.
   \begin{proposition}\label{P:2.1} 
For any $T>0$,  $u_0\in H$, and $\eta\in  L^2(J_T,H)$,     there is  a unique solution $u\in C(J_T, H)\cap L^2(J_T,V)$  of problem \eqref{1.1},~\eqref{0.3}.~It satisfies the following inequalities 
 	\begin{align}
 	\|u(t)\|^2&\le e^{-\nu \la_1 t}\|u_0\|^2+\nu^{-2}\la_1^{-1} \|\eta\|_{L^2(J_T,V')}^2, \quad t\in J_T,  	\label{1.3} \\
		\|u\|_{L^2(J_T,V)}^2&\le  {\nu}^{-1}\|u_0\|^2+  {\nu^{-2}}  \|\eta\|_{L^2(J_T,V')}^2,\label{AAA1.3} 
 	\end{align}      and   the     equality  
 	\begin{equation}\label{1.4}
 		\|u(t)\|^2=e^{-\nu\la_1 t} \|u_0\|^2+2\int_0^te^{-\nu\la_1 (t-s)}\left(\lag\eta(s),u(s)\rag-\nu[u(s)]^2\right)\!\dd s, \quad t\in J_T.
 	\end{equation}

 	  \end{proposition}   
 	Let us   define the mapping 
$$
S_t:H\times L^2(J_T,H)\to H, \quad (u_0,\eta)\mapsto u(t), \quad t\in J_T.
$$  	The following   stability result holds.
 \begin{proposition}\label{P:2.2}
 Assume that the sequence $\{u_0^n\}$ converges   weakly to $u_0$ in $H$ and the sequence $\{\eta^n\}$ converges   strongly to $\eta$ in $L^2(J_T,H)$.~Then
 \begin{align}
 		S_t(u_0^n,\eta^n)\rightharpoonup S_t(u_0,\eta) \quad &\text{weakly in $H$ for $t\in J_T$},\label{E:A}\\
 		S_\cdot(u_0^n,\eta^n)\rightharpoonup S_\cdot(u_0,\eta) \quad &\text{weakly in $L^2(J_T,V)$}\label{E:B}.
 \end{align}
 \end{proposition}
 The proofs of these two  propositions are  carried out by standard methods and are given in the Appendix. 
 
  Let us now describe     the functional  setting in which    Theorem~\ref{T:1.1} is applied.~The space $H$ is defined by \eqref{N:0.9}, $E:= L^2([0,1],H)$, and the mapping  
$$
S:=S_1:H\times E\to H, \quad (u_0, \eta)\mapsto u(1)
$$ is the time-one shift   along trajectories of Eq.~\eqref{1.1}.~Then the   restriction to    integer times of the   solution of~\eqref{1.1},~\eqref{0.3},~\eqref{0.7}    satisfies    
\begin{equation}	\label{SSD1.5}
u_k=S(u_{k-1},\eta_k), \quad k\ge1,	
\end{equation}  where 
  $\{\eta_k\}$ is a sequence of i.i.d. random variables
   as in the decomposability condition in the Introduction.~It is straightforward to see that  Condition~(v)   satisfied.\footnote{Note that  $\KK:=\supp\DD(\eta_k)$ is compact in $E$, since   it is   contained in a Hilbert~cube.}  
In the   next three subsections, we   check Conditions~(i)-(iv).

  \subsection{Regularity condition} \label{SS:2.2}

   The smoothness of the mapping $S:H\times E\to H$ and the boundedness  of its derivatives on bounded subsets of $H\times E$ are proved using well-known methods    (e.g., see  Chapters~I and~VII in~\cite{BV1992} for the case of  bounded domain $D$).
   
Let  us consider the  linearisation~of~Eq.~\eqref{1.1} around the~trajectory $\tilde u(t)=S_t(u,\eta)$ corresponding to an  initial condition~$\tilde u(0)= u\in  H$ and  control~$\eta\in \KK$:
 \begin{gather}
  \dot w + \nu  Lw + Q(\tilde u,w) =0, \label{3.3C}\\
 	w(0)=w_0,   \label{3.3CD}
 \end{gather} where
\begin{equation}\label{3.4C}
Q(a,b)=\Pi (\lag a, \nabla \rag b)+\Pi (\lag b, \nabla \rag a).
\end{equation}Then $(D_u S)(u,\eta)w_0 = w(1)$ for any $w_0\in H$, and we can write  $w=v_1+v_2$, where~$v_1$ and $v_2$ are the solutions of the problems 
\begin{align}
	\dot v_1 +  \nu  L v_1  &=0, \quad v_1(0)=w_0, \label{E:AAA}\\
	\dot v_2 +  \nu Lv_2 + Q(\tilde u,w)& =0, \quad  v_2(0)=0. \label{E:BBB}
\end{align}The   representation \eqref{reps1} holds  with  the linear operators 
 \begin{align*}
	\Psi_1 &:H\to H, \quad  w_0\mapsto v_1(1), \\
\Psi_2 (u,\eta)&:H\to H, \quad  w_0\mapsto v_2(1).
\end{align*}    
 Inequality 
$$
\|v_1(t)\| \le   e^{-\nu \lambda_1 t/2}\|w_0\|
$$ implies   \eqref{reps2} with $\varkappa:=e^{-\nu \lambda_1/2}$.  The following lemma (whose proof  is given in the Appendix)  completes the verification     of~Condition (i). 
\begin{lemma}\label{L:2.4}For   any~$(u,\eta)\in H\times \KK$, 
the operator~$\Psi_2 (u,\eta):H\to H$ is compact.
\end{lemma}

 \subsection{Asymptotic compactness}\label{S:2.2}

   For any integer  $k\ge1$ and any function   $\eta\in L^2(J_k, H)$    of the form~\eqref{0.7} with some sequence $\{\eta_n\}\subset E$,  we    write
	\begin{equation}\label{1.6}
	u(t)=S_t(u_0,\eta)=:S_t(u_0;\eta_1,\ldots,\eta_k), \quad t\in J_k.
\end{equation}
  \begin{proposition}\label{P:2.3}
  For any bounded sequence $\{u_0^n\}$ in $H$, any sequence of integers $l_n\ge1$ such that $l_n\to \ty$, and any family $\{\zeta_m^n: m,n\ge1\}\subset \KK$, the sequence 
  $$
  v_n=S_{l_n}(u_0^n; \zeta_1^n,\ldots,\zeta_{l_n}^n)
  $$ is precompact in $H$.
\end{proposition}
\begin{proof} 
Let us first explain the scheme of the proof.~Using the dissipativity of the system, we   show  that $\{v_n\}$ is bounded in $H$.~As $H$ is a Hilbert space,   there is a subsequence $\{v_{k_n}\}$ such that  
\begin{equation}\label{1.7}
	v_{k_n}\rightharpoonup w  \quad \text{weakly in $H$}.
\end{equation}
Hence, we have  
\begin{equation}\label{1.8}
\|w\|\le \liminf_{n\to \ty} \|v_{k_n}\|.	
\end{equation} 
On the other hand,    using   energy equality~\eqref{1.4} and Proposition~\ref{P:2.2},    we show~that     
\begin{equation}\label{1.9}
\limsup_{n\to \ty} \|v_{k_n}\|\le \|w\|.	
\end{equation}  
Inequalities \eqref{1.8} and \eqref{1.9} together imply that 
$$
	v_{k_n}\to w  \quad \text{strongly in $H$},
$$which gives the required result.  

\smallskip
{\it Step~1:~Boundedness.}  Let us show that, for any bounded set $\HH\subset H$, the attainability set $\aA(\HH)\subset H$ is bounded. Indeed, let        
$$
M_1:=\sup_{u\in \HH}\|u\|^2<\ty. 
$$
For any $u_0\in \HH$ and $\zeta_1, \ldots, \zeta_k\in \KK$, let   $u_k:=S_k(u_0;\zeta_1,\ldots,\zeta_k)$.~Then  by inequality~\eqref{1.3}, we~have   
$$
\|u_k\|^2\le \varkappa \|u_{k-1}\|^2+M_2, 
$$where $\varkappa:=e^{-\nu\la_1 }<1$ and $M_2:=\nu^{-2}\la_1^{-1} \sup_{\eta\in \KK}\|\eta\|_{L^2([0,1],V')}^2$.~Iterating this,~we~get 
$$
\|u_k\|^2 \le  \varkappa^k \|u_0\|^2+M_2(1-\varkappa)^{-1}\le M_1+M_2(1-\varkappa)^{-1}=:M.
$$
This shows that
$$
\sup_{u\in \aA(\HH)}\|u\|^2\le M.
$$

\smallskip
{\it Step~2:~Proof of~\eqref{1.9}.}~From the previous step it follows that $\{v_n\}$ is bounded in $H$.~Let~$\{v_{k_n}\}$ be a   subsequence verifying   \eqref{1.7}. It is  of the~form
$$
v_{k_n}=S_{\ell_n}(u_0^n;\eta_1^n,\ldots,\eta_{\ell_n}^n)
$$
for some vectors  $\eta_1^n,\ldots,\eta_{\ell_n}^n\in \KK$  and  integers $\ell_n\ge1$ such that $\ell_n\to \ty$.~Using the boundedness of the set~$\aA(\{u_0^n, n\ge1\})$ and  
   passing to a subsequence if necessary (applying the diagonal
process), we can assume that 
\begin{equation}\label{E:C}
w_{n,m}:=S_{\ell_n-m}(u_0^n;\eta_1^n,\ldots,\eta_{\ell_n-m}^n)\rightharpoonup w_m  \quad \text{weakly in $H$}
\end{equation}
for any $m\ge1$ and  some $w_m\in H$.~Using the compactness of $\KK$ and again passing to a  subsequence if necessary, we can assume that
$$
\eta^n_{\ell_n-i}\to \xi_{m-i}^m \quad \text{strongly in $E$, $i= 0, \ldots m-1$.}
$$Combining this with \eqref{E:C} and Proposition~\ref{P:2.2},  we~get
$$
  v_{k_n}= S_m(w_{n,m};\eta^n_{\ell_n-m+1}, \ldots,\eta^n_{\ell_n}) \rightharpoonup S_m(w_m;\xi^{m}_1, \ldots,\xi^{m}_m)  \quad \text{weakly in $H$}.
$$From \eqref{1.7} we derive  the equality
$$
w=	S_m(w_m;\xi^{m}_1, \ldots,\xi_{m}^m) \quad \text{for any $m\ge1$}.
$$
This and   \eqref{1.4} imply that 
	\begin{align}\label{E:1.4}
 		\|w\|^2&=\|u^m(m)\|^2\nonumber\\&= e^{-\nu\la_1 m}\|w_m\|^2+2\int_0^me^{-\nu\la_1 (m-s)}\left(\lag\xi^m(s),u^m(s)\rag-\nu[u^m(s)]^2\right)\!\dd s,
 	\end{align}
where
$$ u^m(t)= S_t(w_m,\xi^m), \quad\quad \xi^m(t)=\sum_{k=1}^m \I_{[k-1,k)}(t)\xi^{m}_k(t-k+1), 
 \quad t\in J_m.
$$
On the other hand, again by    \eqref{1.4}, we have  
	\begin{align}\label{EE:E1.4}
 		\|v_{k_n}\|^2&=\!\|u^{n,m}(m)\|^2\nonumber\\&=\!e^{-\nu\la_1 m}\|w_{n,m}\|^2\!+\!2\int_0^m\!\!\!e^{-\nu\la_1 (m-s)}\!\left(\lag\eta^{n,m}(s),u^{n,m}(s)\rag\!-\!\nu[u^{n,m}(s)]^2\right)\!\!\dd s,
 	\end{align}where 
 	$$ u^{n,m}(t)= S_t(w_{n,m},\eta^{n,m}),   \quad \eta^{n,m}(t)=\sum_{k=1}^m \I_{[k-1,k)}(t)\eta_{k_n-m+k}^n(t-k+1),
 \quad t\in J_m.
$$
Note that 
\begin{equation}\label{Ereds}
\int_0^me^{-\nu\la_1 (m-s)} \lag\eta^{n,m}(s),u^{n,m}(s)\rag  \dd s \to \int_0^me^{-\nu\la_1 (m-s)} \lag\xi^m(s),u^m(s)\rag \dd s, 
\end{equation}
as $\eta^{n,m}\to \xi^{m}$ strongly in $L^2(J_m,H)$ and $u^{n,m}\rightharpoonup u^{m}$ weakly in $L^2(J_m,V)$.~Since
$$
\left(\int_0^m e^{-\nu\la_1 (m-s)}[\cdot]^2\dd s\right)^{1/2}
$$is a norm in $L^2(J_m,V)$ which is   equivalent to the original one, the following inequality holds   
$$
 \int_0^me^{-\nu\la_1 (m-s)} [u^m(s)]^2\dd s \le \liminf_{n\to \ty}  \int_0^me^{-\nu\la_1 (m-s)} [u^{n,m}(s)]^2\dd s. 
$$  Combining this with \eqref{E:1.4}-\eqref{Ereds}, we get
\begin{align*}
 \limsup_{n\to\ty}\|v_{k_n}\|^2&\le \limsup_{n\to\ty} \left(e^{-\nu\la_1 m}\|w_{n,m}\|^2\right)+\|w\|^2-e^{-\nu\la_1 m}\|w_m\|^2\\&\le e^{-\nu\la_1 m}M+ \|w\|^2.
\end{align*}As $m\ge1$ is arbitrary, we arrive at \eqref{1.9}.
\end{proof}

  \subsection{Approximate controllability}

  By the decomposability assumption, we have $0\in \KK$.~In view of \eqref{1.3},   
       Condition~(iii) is verified with $\hat u=0$, $\zeta_1=\ldots=\zeta_n=0$, and  sufficiently large~$n\ge1$.

        To check Condition~(iv), we consider the following linearisation~of~Eq.~\eqref{1.1} around the same~trajectory $\tilde u$ as in Section~\ref{SS:2.2}:
 \begin{gather}
  \dot w + \nu  Lw + Q(\tilde u,w) =   \zeta(t), \label{3.3}\\
 	w(0)=0,   \nonumber
 \end{gather} where $Q$ is given by \eqref{3.4C}.~Then $(D_\eta S)(u,\eta)\zeta = w(1)$ for any $\zeta\in E$.~For   any   smooth function~$w_1\in H $,~we can find a smooth func\-tion~$w:[0,1]\times D\to \R^2$ such that $w(0)=0$, $w(1)=w_1$, and $w(t)\in H$ for all $t\in [0,1]$. Replacing~$w$~into Eq.~\eqref{3.3},~we find explicitly  a control~$ \zeta \in E$ such that~$(D_\eta S)(u,\eta)\zeta = w_1$. This~shows that the image of the mapping~$(D_\eta S)(u,\eta):E\to H$ is dense~in~$H$ for any~$(u,\eta)\in H\times \KK$.   Thus Conditions~(i)-(v) are verified. Applying Theorem~\ref{T:1.1}, we complete the proof of the Main Theorem.

 \section{Appendix}
 \subsection{Proof of Proposition~\ref{P:2.1}}

 The  existence and uniqueness of solution is proved, e.g., in  Chapter~III of~\cite{MR0609732}.   Here we give a   formal derivation of inequalities~\eqref{1.3} and \eqref{AAA1.3}  and equality~\eqref{1.4}.  

   Taking the scalar product in $H$ of Eq.~\eqref{1.1} with~$2u$ and   using the identity
$$
\lag B(u),u\rag=0, \quad u\in V,
$$  
we get 
\begin{equation}\label{E:ZE}
\frac{\dd}{\dd t}\|u\|^2	+2\nu \|u\|_1^2=2\lag\eta,u\rag\le 2 \|\eta\|_{V'}\|u\|_1\le \nu \|u\|_1^2+ \nu^{-1}\|\eta\|_{V'}^2.
\end{equation}Combining this with the Poincar\'e inequality, we obtain \eqref{1.3}.~From \eqref{E:ZE} we also derive the inequality
\begin{equation}\label{E:zds}
	\int_0^t \|u\|_1^2\dd s\le  {\nu}^{-1}\|u_0\|^2+  {\nu^{-2}}  \int_0^t  \|\eta\|_{V'}^2\dd s\quad t\in J_T 
\end{equation}which implies \eqref{AAA1.3}.~To   prove \eqref{1.4}, we rewrite the   equality in \eqref{E:ZE} in the~form
$$
\frac{\dd}{\dd t}\|u\|^2	+\nu\la_1\|u\|^2=2\left(\lag \eta,u\rag-\nu [u]^2\right)
$$and apply   the variation of constants formula.

 \subsection{Proof of Proposition~\ref{P:2.2}}

  By Proposition~\ref{P:2.1}, the sequence $u^n=S_\cdot(u_0^n,\eta^n)$ is bounded in the space $C(J_T,H)\cap L^2(J_T,V)$. From the inequality 
  $$
  \|B(u)\|_{V'}\le C \|u\| \|u\|_1, \quad u\in V 
  $$ and Eq.~\eqref{1.1} we derive that 
  \begin{equation}\label{E:MMM} 
  	\|\dot u^n\|_{L^2(J_T,V')} \le M_1, \quad n\ge1.
  \end{equation}
For any $R>0$, let us set $D_R:=D\cap B_{\R^2}(R)$.~Then  the space $H^1(D_R)$ is compactly embedded into $L^2(D_R)$.~Applying Theorem~2.1 of Chapter~III in~\cite{MR0609732} with spaces~$X_0=H^1(D_R)$, $X=L^2(D_R)$, $X_1=H^{-1}(D_R)$  and numbers~$\alpha_0=\alpha_1=2$, and using   the diagonal
process,
we find a subsequence~$\{u^{k_n}\}$ such~that
  \begin{align}
  	&u^{k_n}\rightharpoonup u \quad &\text{weak-star in $L^\infty(J_T,H)$}, \label{AAAZ1}\\
  	&u^{k_n}\rightharpoonup u \quad &\text{weakly in $L^2(J_T,V)$},\label{AAAZ2}\\
  	&u^{k_n}\to u\quad &\text{strongly in $L^2(J_T,L^2(D_R))$} \label{AAAZ3}
  \end{align}for any $R>0$.~Passing to the limit in the equation  for $u^{k_n}(t)$, we   conclude that~$u(t)$ is the solution~$S_t(u_0,\eta)$, $t\in J_T$.~Moreover, by the uniqueness of the limit, it is easy to see that limits~\eqref{AAAZ1}-\eqref{AAAZ3} hold for the full sequence $\{u^n\}$. This proves~\eqref{E:B}. 
  
Let us take any  $\varphi \in \VV$ (see~\eqref{N:0.8}). By inequality \eqref{1.3}, we~have 
$$
 |\lag u^n(t),\varphi\rag |\le M_2, \quad n\ge1, \, t\in J_T,
$$and by inequality~\eqref{E:MMM},   
\begin{align*} 
|\lag u^n(t+\tau)-u^n(t),\varphi\rag |&\le \int_t^{t+\tau} |\lag \dot u^n(s),\varphi\rag|\dd s	\\&\le \tau^{\frac12} \|\dot u^n\|_{L^2(J_T,V')} \|\varphi\|_1
\le \tau^{\frac12} M_1 \|\varphi\|_1
\end{align*}for any $t\in J_T$ and $\tau\in (0,T-t)$. Thus   the   Arzel\`a--Ascoli theorem implies that  
$$
 \lag u^n (t),\varphi \rag  \to \lag u (t),\varphi \rag  
$$uniformly in $t\in J_T$. Using the fact that $\VV$ is dense in $H$, we get \eqref{E:A}.

\subsection{Proof of Lemma~\ref{L:2.4}}
 
 Let us set  $J:=[0,1]$.~By Proposition~\ref{P:2.1},  we have $\tilde u=S_\cdot(u,\eta)\in L^2(J, V)$ for any~$u\in H$ and $\eta\in \KK$.~Using standard methods, one can show that the     mapping~$\Psi:w_0\mapsto w$ (i.e., the resolving operator of problem~\eqref{3.3C},~\eqref{3.3CD}) is continuous from~$H$ to~${\cal X}:=L^2(J,V)\cap W^{1,2}(J,H^{-1}) $. Let us show that     the linear mapping 
 $$
 \Phi^{\tilde u}: H\to   L^2(J,H^{-1}), \quad w_0\mapsto \lag \tilde u,\nabla\rag w+\lag w,\nabla\rag \tilde u
 $$is   compact.~Indeed, let $\{w_0^n\}$ be a bounded sequence  in $H$.~Then the sequence $\{\Psi(w_0^n)\}$ is bounded in $\cal X$.~As in the previous subsection, we apply   Theorem~2.1~of Chapter~III in~\cite{MR0609732}  with  spaces~$X_0=H^1(D_R)$, $X=L^2(D_R)$, $X_1=H^{-1}(D_R)$ and numbers~$\alpha_0=\alpha_1=2$, and use  the diagonal
process  to 
  find a subsequence~$\{\Psi(w_0^{k_n})\}$   converging strongly in~$L^2(J,L^2(D_R))$ for any $R>0$.

 On the other hand, for any $\e>0$, there is a number~$R>0$ and a smooth function $\varphi: J\times D\to \R^2 $  with $\supp \varphi(t,\cdot)\subset D_R $ for any $t\in J$  and 
 \begin{equation}\label{EEASDF}
 	\|\tilde u-\varphi\|_{L^2(J,V)}<\e.
 \end{equation}
Using the boundedness of the sequence $\{\Phi^{\tilde u}(w_0^{k_n})\}$ in $L^2(J,V)$ and the estimate 
 $$
\| \lag  a,\nabla\rag b\|_{H^{-1}}\le C_1 \|a\|_{H^{1/2}}\|b\|_{H^{1/2}}\le C_2 \|a\|_1 (\|b\|\|b\|_1)^{\frac12}, \quad a,b\in V, 
 $$   we obtain
 \begin{align*}
\| \Phi^{\tilde u}(w_0^{k_n})-\Phi^{\tilde u}(w_0^{k_m})\|_{L^2(J,H^{-1})}&\le 	\|\Phi^{\varphi}(w_0^{k_n})-\Phi^{\varphi}(w_0^{k_m})\|_{L^2(J,H^{-1})}\\&\quad+	\|\Phi^{\tilde u-\varphi}(w_0^{k_n})-\Phi^{\tilde u-\varphi}(w_0^{k_m})\|_{L^2(J,H^{-1})}\\&\le C_3 \| \Psi(w_0^{k_n})-\Psi(w_0^{k_m})\|^{1/2}_{L^2(J,L^2(D_R))}\\&\quad+	C_3\| \tilde u-\varphi \|_{L^2(J,V)},
 \end{align*}  where the constant $C_3>0$ does not depend on the numbers $n,m,\e$. Combining this with \eqref{EEASDF} and choosing $n$ and $m$ sufficiently large, we see that
 $$
 \| \Phi^{\tilde u}(w_0^{k_n})-\Phi^{\tilde u}(w_0^{k_m})\|_{L^2(J,H^{-1})}<2C_3\e. 
 $$
  This shows that $ \Phi^{\tilde u}$ is compact.~Thus  the map~$\Psi_2  :H\to H$ is compact as  a composition of $ \Phi^{\tilde u}$ with some linear continuous~map.

\def\cprime{$'$} \def\cprime{$'$}
  \def\polhk#1{\setbox0=\hbox{#1}{\ooalign{\hidewidth
  \lower1.5ex\hbox{`}\hidewidth\crcr\unhbox0}}}
  \def\polhk#1{\setbox0=\hbox{#1}{\ooalign{\hidewidth
  \lower1.5ex\hbox{`}\hidewidth\crcr\unhbox0}}}
  \def\polhk#1{\setbox0=\hbox{#1}{\ooalign{\hidewidth
  \lower1.5ex\hbox{`}\hidewidth\crcr\unhbox0}}} \def\cprime{$'$}
  \def\polhk#1{\setbox0=\hbox{#1}{\ooalign{\hidewidth
  \lower1.5ex\hbox{`}\hidewidth\crcr\unhbox0}}} \def\cprime{$'$}
  \def\cprime{$'$} \def\cprime{$'$} \def\cprime{$'$}

\end{document}